\documentclass{article}
\usepackage{spconf,amsmath,amsfonts,amssymb,amsthm, graphicx}
\usepackage{multirow}
\usepackage{rotating}
\usepackage{url}
\usepackage{color}
\usepackage{enumerate}

\def\kS{$k$-space }
\def\C{\mathbb{C}}
\def\A{\boldsymbol{A}}
\def\At{\tilde{\boldsymbol{A}}}

\makeatletter
\def\captionof#1#2{{\def\@captype{#1}#2}}
\makeatother
\newtheorem{thm}{Theorem}
\newtheorem{proposition}{Proposition}

\title{Variable density Compressed Sensing in MRI. \\Theoretical vs heuristic sampling strategies.}
\name{Nicolas Chauffert$^{(1,2)}$, Philippe Ciuciu$^{(1,2)}$, Pierre Weiss$^{(3,4)}$}
\address{$^{(1)}$CEA/DSV/I$^2$BM NeuroSpin center, B\^at. 145, F-91191 Gif-sur-Yvette, France\\
$^{(2)}$ INRIA Saclay Ile-de-France, Parietal team, 91893 Orsay, France.\\
$^{(3)}$ Institut des Technologies Avancées du Vivant~(CNRS UMS 3039), 31106 Toulouse, France.\\
$^{(4)}$ Institut de Mathématiques de Toulouse (CNRS UMR~5219), 31062 Toulouse Cedex 9, France.\\\normalsize{\url{nicolas.chauffert@inria.fr} , \url{philippe.ciuciu@cea.fr}, \url{pierre.armand.weiss@gmail.com}} }

\begin{document}
\ninept
\maketitle
\begin{abstract}
The structure of Magnetic Resonance Images~(MRI) and especially their compressibility in an appropriate representation basis enables the application of the compressive sensing theory, which guarantees exact image recovery from incomplete measurements. According to recent theoretical conditions on the reconstruction guarantees, the optimal strategy is to downsample the \kS using an independent drawing of the acquisition basis entries. Here, we first bring a novel answer to the synthesis problem, which amounts to deriving the optimal distribution (according to a given criterion) from which the data should be sampled. Then, given that the sparsity hypothesis is not fulfilled in the $k$-space center in MRI, we extend this approach by densely sampling this center and drawing the remaining samples from the optimal distribution. We compare this theoretical approach to heuristic strategies, and show that the proposed two-stage process drastically improves reconstruction results on anatomical MRI.
\end{abstract}
\begin{keywords}
MRI, Compressive sensing, wavelets, synthesis problem, variable density random undersampling.
\end{keywords}
\section{Introduction}
\label{sec:intro}
Decreasing scanning time is a crucial issue in MRI since it could increase patient comfort,
improve image quality by limiting the patient's movement and reducing geometric distortions and make the exam cost cheaper. A simple way to reduce acquisition time consists of acquiring less data samples by downsampling the $k$-space.
Compressed Sensing~(CS) theory \cite{Candes06,Donoho06b} gives guarantees of recovering a sparse signal from incomplete measurements in an acquisition basis. These guarantees depend on the sparsity of the image in a representation basis~(e.g., wavelet basis) and on the mutual properties of the acquisition and representation bases.

Nevertheless, the first CS theories did not provide the right answer to the synthesis problem, which consists of deriving an optimal sampling pattern. Recent results introduced in~\cite{Rauhut10,CandesP11} propose a new vision of the CS theory relying on an independent sampling of measurements. Since the proposed upper bound gives the sharpest condition on the number of measurements required to guarantee exact recovery of sparse signals, optimal sampling schemes in CS result from an independent drawing of the selected samples. In this paper, we call \textit{theoretical optimal distribution} a distribution which provides optimal reconstruction guarantees.

In this paper, we propose an answer to the synthesis problem in the MRI framework, which amounts to finding the optimal downsampling pattern of the \kS\!\!. Following~\cite{Rauhut10}, we derive the optimal distribution from which the data samples are independently drawn~(Section~\ref{sec:optimal_sampling}). Then, according to the observation that MRI images are not sparse in the low frequencies, we propose a two-stage sampling process: first, a given area of the \kS center is completely sampled in order to recover the low frequencies, and then the existing theory is applied to the remaining high frequency content of MRI images for which the sparsity assumption is more tenable~(Section~\ref{sec:spars}). In Section~\ref{sec:results}, we compare our method with the state-of-the-art~\cite{Lustig07}, and recover very close reconstruction results, while the proposed approach more likely meets the required sparsity assumption for designing CS sampling schemes. Our results show that the proposed two-stage strategy drastically improves reconstruction results.

\section{NOTATION}
\label{sec:notation}

In this paper, we consider a discrete 2D \kS composed of $n$ pixels, in which the MRI signal
is acquired. The acquisition of the complete \kS gives access to the Fourier transform of a reference image, to which we will compare our reconstruction results. We assume that the image is sparse in a representations basis~(e.g., a wavelet basis) as pointed out in~\cite{Lustig07}, and we denote by $x$ this sparse representation. Let $\Psi_1 \dots \Psi_n \in \C^n$ be the wavelet atoms of an orthogonal wavelet tranform and $\boldsymbol{\Psi} = [\Psi_1, \dots , \Psi_n]\!\in\!\C^{n \times n}$.
The MRI image then reads $\boldsymbol{\Psi}x$. Let us introduce $\boldsymbol{F}^*$ the Fourier transform and $\A_0=\boldsymbol{F}^* \boldsymbol{\Psi}$ the orthogonal transform between the {\em representation}~(the wavelet basis) and the {\em acquisition}~(the \kS\!\!) domains.
$\A_0\!\in\!\C^{n \times n}$ is then an orthogonal matrix. Finally, we will denote by
$\A\!\in\!\C^{m \times n} $ a matrix composed of $m$ lines of $\A_0$. $\A$ is called the acquisition matrix, and $m$ represents the number of Fourier coefficients actually measured.
We define the $\ell_1$ problem as the convex optimization problem that consists of finding the vector of minimal $\ell_1$-norm subject to the constraint of matching the observed data $y=\A x $:
\begin{equation}
\label{eq:l1}
 \underset{\A w=y}{\operatorname{argmin }} \|w\|_1
\end{equation}

\section{OPTIMAL SAMPLING IN AN ORTHOGONAL SYSTEM}\label{sec:optimal_sampling}
\subsection{Optimal sampling distribution}

First, we summarize a result recently introduced by Rauhut~\cite{Rauhut10}.
Let $P=(P_1, \dots, P_n)$ be a discrete probability measure. Let us define the scalar product
$\langle .,.\rangle_P$ by $\langle x,y \rangle_P=\sum_{i=1}^n x_i \bar{y_i} P_i$. 
Matrix $\tilde{\A}_0$ is defined by $(\tilde{\A}_0)_{ij}=(\A_0)_{ij}/P_i^{1/2}$, i.e.
the lines of $\A_0$ are normalized by $P^{1/2}$. Then, for all possible distribution $P$,
$\tilde{\A}_0$ has orthogonal columns with repect to $\langle .,.\rangle_P$. Following~\cite{Rauhut10},
the infinite norm of $\tilde{\A}_0$ is given by:
$K(P)=\underset{1 \leqslant i,j \leqslant n}{\operatorname{sup }} |(\tilde{\A}_0)_{ij}|$.
Rauhut's result~\cite[Theorem~4.4]{Rauhut10} links the number of required measurements $m$ to
the sparsity level $s$ of \emph{any} unknown signal in order to guarantee its exact recovery from an independent
sampling of its Fourier coefficients:
\begin{thm}{\cite[Th.~4.4]{Rauhut10}}
\label{thm:Rauhut}
Consider a sequence of $m$ i.i.d. indexes drawn from the law $P$, and denote $\A\!\in\!\C^{m \times n}$ the matrix composed of lines of $\tilde{\A}_0$ corresponding to these indexes. Assume that,
\begin{eqnarray}
\label{eq:borne1}
\frac{m}{\ln(m)} & \geqslant & C K(P)^2 s \ln^2(s) \ln(n) \\
\label{eq:borne2}
m & \geqslant & D K(P)^2 s \ln(\epsilon^{-1})
\end{eqnarray}
Then, with probability $1-\epsilon$, every $s$-sparse vector $x \in \C^n$ can be recovered from
observations $y=\A x$, by solving the $\ell_1$ minimization problem in Eq.~\eqref{eq:l1}. The values $C$ and $D$ are universal constants.
\end{thm} 

\noindent Let us denote $a_i^*$ the $i$-th line of $\A_0$. We show the following result:

\begin{proposition}
Since $K(P)= \underset{1 \leqslant i,j \leqslant n}{\operatorname{sup}}
|(\tilde{\A}_0)_{ij}| =\underset{1 \leqslant i \leqslant n}{\operatorname{sup }}
|\frac{\|a_i\|_\infty}{P_i^{1/2}}|$,
\begin{enumerate}[(i)]
\item the optimal distribution $\pi$ that minimizes the upper bound in~\eqref{eq:borne1}--\eqref{eq:borne2} is:
$\pi_i=\|a_i\|_\infty^2/L$ where $L=\sum_{i=1}^n \|a_i\|_\infty^2$. 
\item $K(\pi)^2=L$.\\
\end{enumerate}
\end{proposition}
\begin{proof}
\begin{enumerate}[(i)]
\item[] 
\item Taking $P=\pi$, we get $K(\pi)=\sqrt{L}$. Now assume that $q\neq \pi$, since $\sum_{k=1}^n q_k = \sum_{k=1}^n \pi_k=1$, $\exists j \in [1,n]$ such that $q_j<\pi_j$. Then $K(q)\geqslant \|a_j\|_\infty/q_j^{1/2} > \|a_j\|_\infty/\pi_j^{1/2} = \sqrt{L} = K(\pi)$. So, $\pi$ is the distribution that minimizes $K(P)$.
\item is a consequence of $\pi$'s definition.
\end{enumerate}
\end{proof}
\noindent In the next part, we assess the upper bound in inequality~\eqref{eq:borne1}.

\subsection{Discussion on the upper bound}

In \cite{Puy11,CandesP11}, or even in \cite[Th.~4.2]{Rauhut10}, a $O(s \log (n))$ upper bound has been derived.
Nevertheless, these results only consider the probability to recover \emph{a given} sparse signal. The result introduced in Theorem~\ref{thm:Rauhut} gives a \emph{uniform} result and thus is more general since it enables to apply the CS theory to \emph{all} sparse signals and for our concern, to all MRI images. Moreover, it is more general than the one proposed in~\cite{Puy11}, since no assumption on the sign of the sparse signal entries is needed. Roughly speaking, the upper bound in Eq.~\eqref{eq:borne1} shows that the number of measurements needed to perfectly recover an $s$-sparse signal is $O(s \log^4(n))$ (since $m\leqslant n$ and $s\leqslant n$). According to~\cite{Rauhut10}, this bound is the best known result for uniform recovery.

Nevertheless, this bound is not usable in practice to determine the number of measurements.
Indeed, for a 2D $256 \times 256$ image, $\log^4(n)=15128$, $K^2(\pi)=L$ only depends on the choice of the wavelet representation and $L\approx 10$: In our experiments, we used Symmlets with 10 vanishing moments and got $L=8.34$. Rauhut~\cite{Rauhut10} suggests that $C \gg 1$, which actually makes the upper bound unusable in Eq.~\eqref{eq:borne1}.

Recent results~\cite{Juditsky11b} give $O(s^2)$ bounds for the number of measurements needed to guarantee exact reconstruction. Nevertheless, the constants are lower and guarantee the reconstruction of very sparse signals. Unfortunately, the $O(s^2)$ bound called quadratic bottleneck is a strong limit for applicability in large scale scenarii. 

\subsection{Rejection of samples}

An important issue of this theoretical approach is that \kS positions that appear more likely according to $\pi$ are drawn more than once. To select a given position at most once, we introduce an intuitive alternative to the solution proposed in~\cite{Puy11}, which consists of rejecting samples associated with $k$-space positions that have already been visited.
 Let us show that this strategy is an improvement over the one proposed in Theorem~\ref{thm:Rauhut}.\\
Let $\A\!\in\!\C^{m \times n}$ denote the matrix composed of $m$ lines of $\At_0$, corresponding to $m$ independent drawings. $\A$ has $m_1$ different lines ($m_1 \leqslant m$), and $\A_1\!\in\!\C^{m_1 \times n}$ denotes the corresponding matrix. Let $\A_2\!\in\!\C^{m \times n}$ denote the matrix obtained after $m_2$ additional independent drawings, where $m_2$ is the smallest integer such that the actual number of different samples matches $m$ exactly. Then:
\begin{proposition}
\label{prop:rejet}
If $x$ is the unique solution of the following $\ell_1$ problem: $\underset{\A w=\A x}{\operatorname{argmin }} \|w\|_1$, it is also the unique solution of $\underset{\A_2 w=\A_2 x}{\operatorname{argmin }} \|w\|_1$
\end{proposition}

\begin{proof}
Assume that $x^*$ fulfills $\A_2 x^*=\A_2 x$ and $\|x^*\|_1 \leqslant \|x\|_1$, then since all lines of $\A$ are lines of $\A_2$, we get $\A x^*=\A x$. Because $x$ is the unique solution of $\underset{\A w=\A x}{\operatorname{argmin }} \|w\|_1$ and $\|x^*\|_1 \leqslant \|x\|_1$, then $x^*=x$ and $x$ is also the unique solution of: $\underset{\A_2 w=\A_2 x}{\operatorname{argmin }} \|w\|_1$.
\end{proof}
In Theorem~\ref{thm:Rauhut}, $m$ is the number of drawings. Prop.~\ref{prop:rejet} proves that the result still holds if $m$ is the number of \emph{different samples} drawn according to law $P$. 

\subsection{Preliminary results}\label{subsec:prelim_results}
The proposed sampling strategy was tested on the Fourier tranform of a refe\-rence image~(Fig.~\ref{fig:prelim}(a)) using several sampling patterns. We compare the above mentioned independent drawing from $\pi$ distribution shown in Fig.~\ref{fig:prelim}(b)) to polynomial distributions $P(x,y)=(1-(\sqrt{2}/n)\sqrt{k_x^2+k_y^2})^p$
for variable power of decay $p=1\!:\!6$ and $-n/2<k_x,k_y \leqslant n/2 $~\cite{Lustig07}. 
In our experiments, the \kS is downsampled at rate 5 meaning that only 20\% of the Fourier coeffients are measured; see Fig.~\ref{fig:prelim}(c)-(d).

\begin{figure}[!h]
\begin{center}
\begin{tabular}{@{}cc}
{\small (a)} & {\small (b)} \\
\includegraphics[height=0.35\linewidth,keepaspectratio=true]{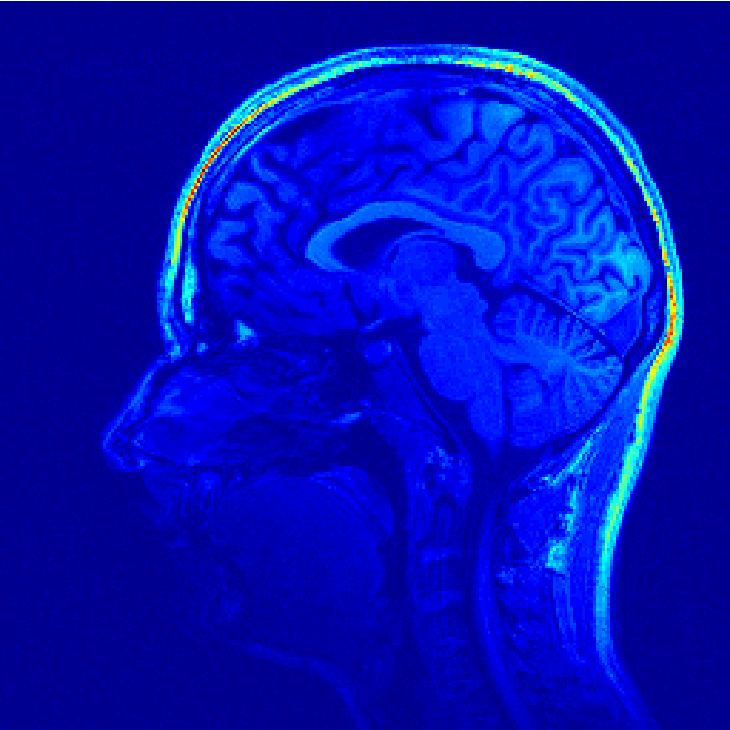}&
\includegraphics[height=0.35\linewidth,keepaspectratio=true]{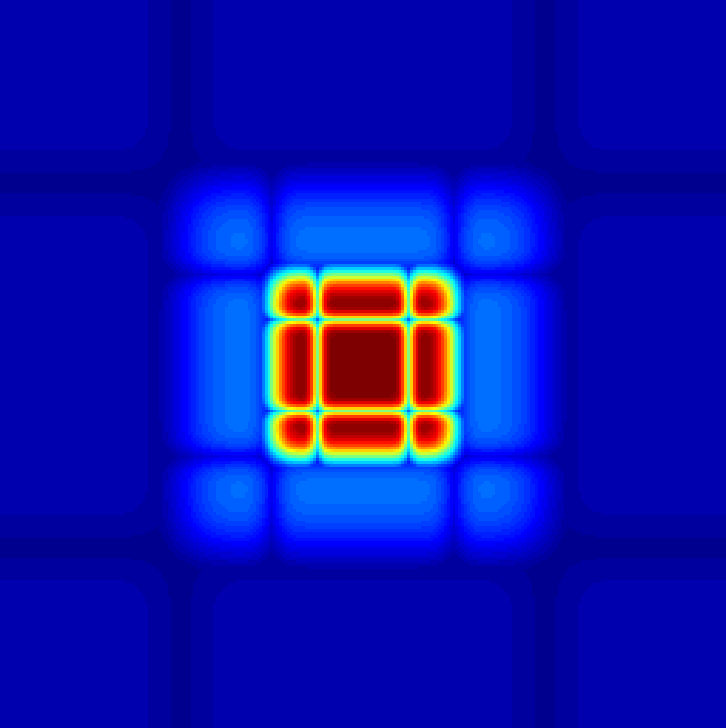}\hspace{2mm}
\includegraphics[height=0.35\linewidth,keepaspectratio=true]{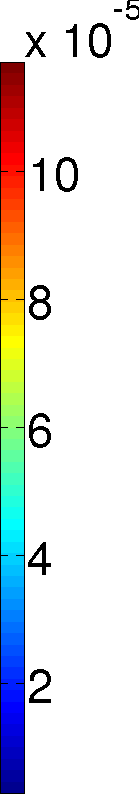} \\
{\small (c)} & {\small (d)}\\
\includegraphics[height=0.35\linewidth,keepaspectratio=true]{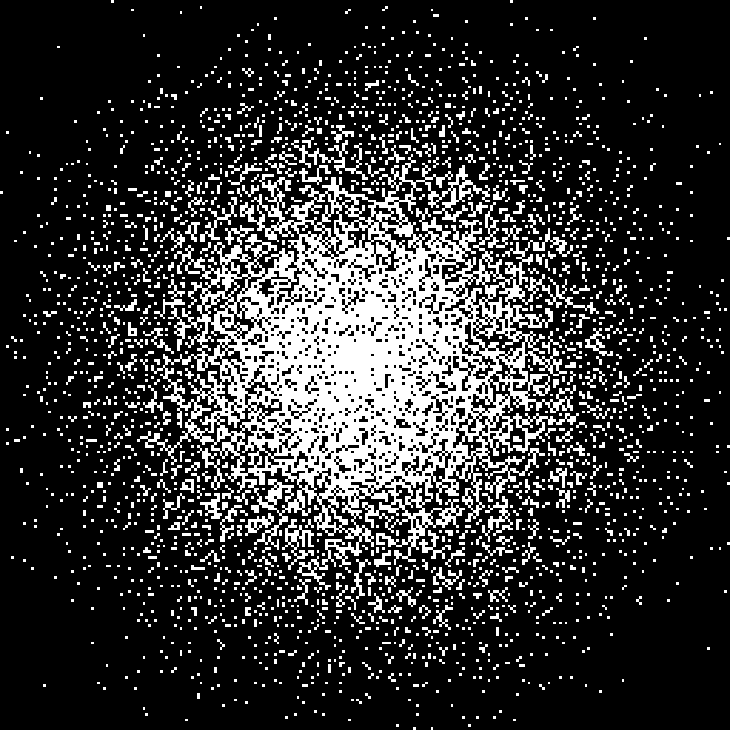}&
\includegraphics[height=0.35\linewidth,keepaspectratio=true]{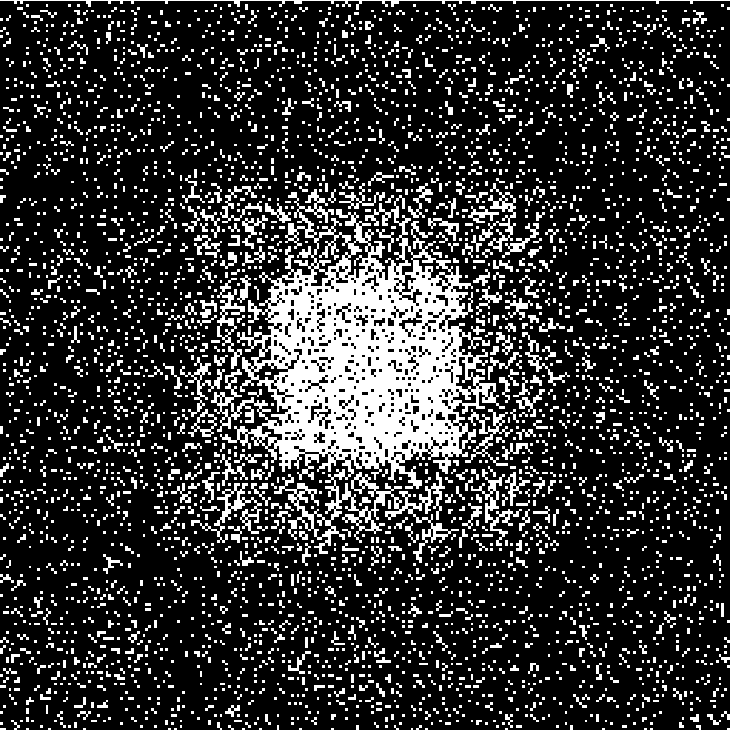}
\end{tabular}\vspace*{-.65cm}
\end{center}
\caption{\label{fig:prelim} \textbf{Example of \kS sampling schemes: selected samples appear in white color.} (a):
Reference $256\times 256$ used in our experiments. (b): Optimal distribution $\pi$. (c):
Sampling pattern based on an independent drawing from a 4th order polynomial density~\cite{Lustig07}. (d): Sampling pattern based on an independent drawing from $\pi$.}
\end{figure}

Since our sampling schemes involve randomness, we performed a Monte-Carlo study and
generated 10 sampling schemes from wich we perform reconstruction by solving the $\ell_1$ minimization problem.
For comparison purposes, we computed method-specific average va\-lue of peak~Signal-to-Noise Ratio~(PSNR) over the 10 reconstruction results as well as the corresponding standard deviations. The results are summarized in Tab.~\ref{tab:prelim} where it is shown that an independent drawing from $\pi$ gives worse results than those obtained with the empirical polynomial distribution with a power $p=4\!:\!6$. Optimality of exponents $p=5,6$ are in agreement with previous work~\cite{Knoll10}.
This poor reconstruction performance can be justified by the fact that MRI images are not
really sparse in the wavelet domain since a lot of low-frequency wavelet coefficients are non-zero.
This lack of sparsity justifies the weakness of this theoretical approach in comparison with polynomial or more generally variable density samplings.

The non-sparsity of the low frequency image content in the wavelet domain means that a
lot of image information is contained around the $k$-space center. This explains why high order polynomial distributions provide better reconstruction results.
In what follows, we propose a novel two-step sampling process to overcome this limitation.


\vspace*{-.5cm}
\begin{table}[!h]
\caption{\label{tab:prelim} Comparison of the reconstruction results in terms of PSNR for various \kS downsampling methods. {\bf Bold font} indicate the best performance with respect to~(wrt) the PSNR and its Std. Dev.}
\vspace*{-0.2cm}
\begin{center}
\begin{tabular}{|c|c||c|c|}
\hline
\multicolumn{2}{|c||}{Sampling density} & Mean PSNR (dB)& Std. dev.\\ 
\hline
\hline
\multirow{6}{1cm}{\begin{sideways}\parbox{15mm}{Polynomial decay.\\ Exponent:}\end{sideways}}
& 1 & 23.55 & 1.40 \\ \cline{2-4}
& 2 & 29.40 & 2.48 \\ \cline{2-4}
& 3 & 32.00 & 3.01 \\ \cline{2-4}
& 4 & 35.52 & 0.57 \\ \cline{2-4}
& 5 & {\bf 36.09} & 0.14 \\ \cline{2-4}
& 6 & {\bf 35.94} & 0.04 \\ \hline \hline
\multicolumn{2}{|c||}{$\pi$} & 33.38 & 2.26 \\  \hline
\end{tabular}
\end{center}
\vspace*{-0.1cm}
\end{table}

\section{IMPROVING IMAGE SPARSITY}
 \label{sec:spars}

Most CS theories are based on sparsity of the signal of interest in a transform basis.
Fig.~\ref{fig:wav} shows that low frequency wavelet coefficients contain a lot of information and does not fulfill the sparsity hypothesis.
The following method tends to decompose the wavelet representation in two parts, one dedicated to low frequencies and the other to high frequencies in which the MRI image is more sparse.
Since low frequency wavelets impact the $k$-space center, their recovery is performed by fully sampling this center. High frequencies are then reconstructed using CS theory.

\begin{figure}[!h]
\begin{center}
\includegraphics[width=0.35\linewidth]{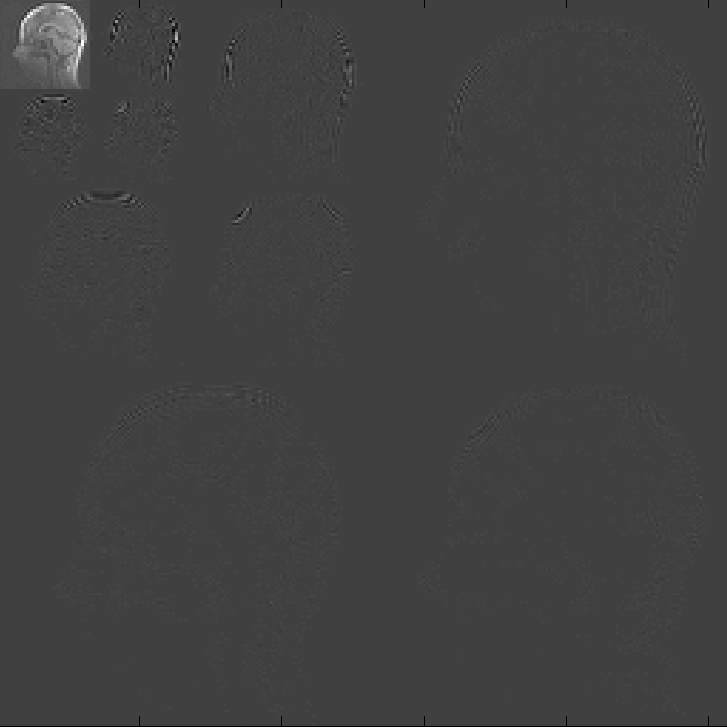}
\end{center}\vspace*{-.65cm}
\caption{\label{fig:wav} Representation of the typical MRI image shown in Fig.~\ref{fig:prelim}(a) in a wavelet
basis (Symmlets with 10 vanishing moments).}
\end{figure}

\subsection{Acquiring the \kS\!\! center: a sparsifying technique.}

Let $n_1$ be the number of low frequency wavelets. Without loss of generality, assume that $\Psi_1 , \dots, \Psi_{n_1}$
are the low frequency wavelet atoms. Let $\displaystyle \Omega=\bigcup_{1\leqslant i \leqslant n_1} \operatorname{supp } (\boldsymbol{F}^*\Psi_i)$. 
By definition of $\Omega$, let us introduce vector $x_\Omega=\boldsymbol{\Psi}^* \boldsymbol{F} y_\Omega$ where $y:=\A_0 x$  and for $ 1\leqslant i \leqslant n$, ${y_{\Omega}}_i=y_i 1_\Omega(i)$. 
Then $x_i={x_\Omega}_i$ for $ 1\leqslant i \leqslant n_1$ and vector $x_{\Omega^\perp}=x-x_\Omega$ is sparse since it contains no low frequency wavelet coefficient.

\subsection{A two-stage reconstruction}

The signal of interest $x_{\Omega^\perp}$ is now more sparse. We adopt the same strategy as in Section~\ref{sec:optimal_sampling}: we draw samples according to $\pi$ and perform rejection if the sample drawn is located within $\Omega$. Then, we recover the signal $x_{\Omega^\perp}$ by solving $\small\underset{\A w=y-\A x_\Omega}{\operatorname{argmin }} \|w\|_1\normalsize$. We notice that:
\begin{enumerate}[(i)]
\item Since we reject samples already drawn, we can sample from the law $\pi^\star$ defined by $\pi^\star_i=\|a_i\|_\infty^2/L^\star$ if $i \not\in \Omega$, $0$ otherwise, with $L^\star=\sum_{i\notin \Omega}\|a_i\|_\infty^2$. The profile of $\pi^\star$ is shown in Fig.~\ref{fig:K}(b). Amongst the remaining frequencies, the more likely \kS positions are the lower frequencies. High frequencies remains unlikely and will be rarely visited by these schemes.

\item For the particular case of Shannon wavelets, $\small \Omega^\perp= \bigcup_{n_1<j\leqslant n}$\\ ${\rm supp} (\boldsymbol{F}^* \Psi_i)\normalsize$. Then, matrix $\A_{\Omega^\perp}$ (composed of lines of $\At_0$ corresponding to frequencies included in $\Omega^\perp$) has orthogonal co\-lumns. Then, Theorem.~\ref{thm:Rauhut} can be applied with the optimal distribution $\pi^\star$ and $\Omega=\sqrt{L^\star}$. Bounds~\eqref{eq:borne1}--\eqref{eq:borne2} are thus improved since $L^\star< L$.
\end{enumerate}



\vspace*{-.4cm}
\begin{figure}[htp]
\begin{center}
\begin{tabular}{cc}
{\small (a)}& {\small (b)}\\
\includegraphics[height=0.35\linewidth]{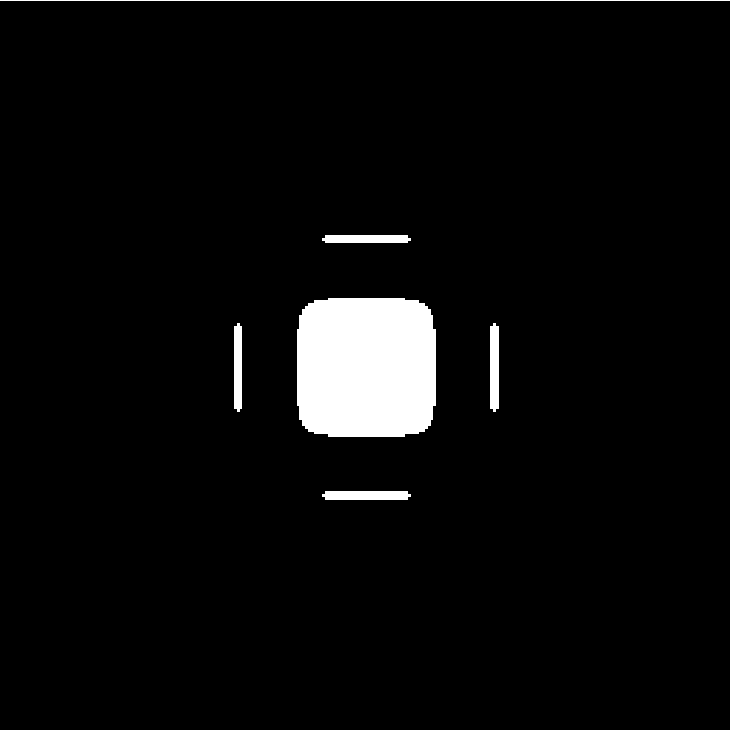}&
\includegraphics[height=0.35\linewidth]{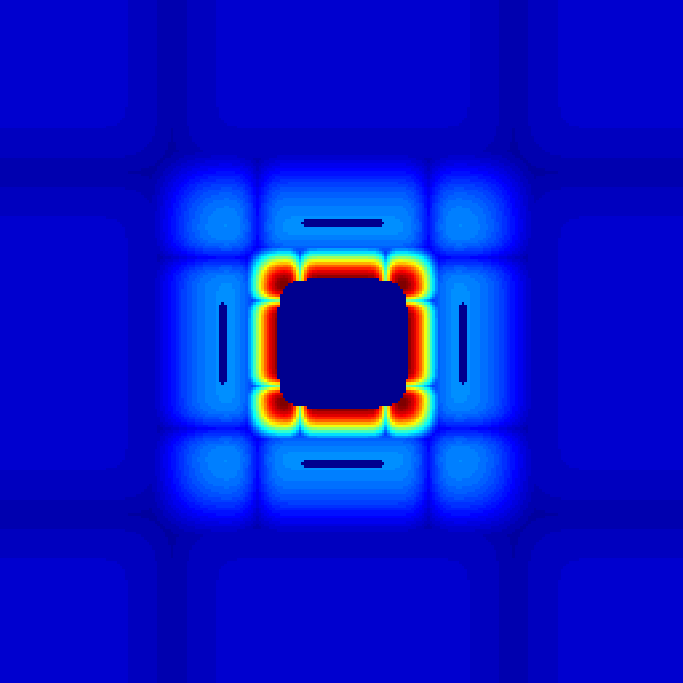}\hspace*{2mm}
\includegraphics[height=0.35\linewidth]{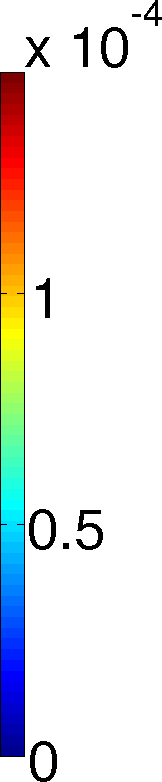}\\
\end{tabular}\vspace*{-.65cm}
\end{center}
\caption{\label{fig:K} Representation of $\Omega$ for a Symmlet-10 transform~(a) and distribution $\pi^\star$ (b).}
\end{figure}

\section{RESULTS}\label{sec:results}

Here, we compare several classical sampling schemes used in MRI with our two-stage approach
either combined with high frequency sampling from $\pi^\star$-distribution or from
polynomial densities, as illustrated in Fig.~\ref{fig:sampling_schemes}. For each scheme, the number of samples acquired represents only 20\% of the full $k$-space.

\begin{figure}[!h]
\begin{center}
\begin{tabular}{@{}ccc}
{\small (a)}&{\small (b)}&{\small (c)} \\
\includegraphics[height=0.3 \linewidth]{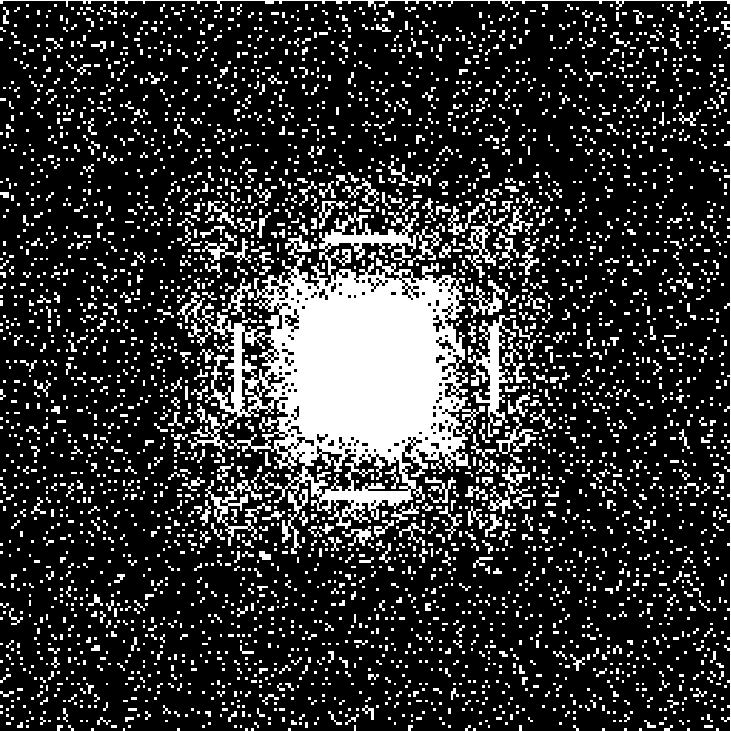}&
\includegraphics[height=0.3 \linewidth]{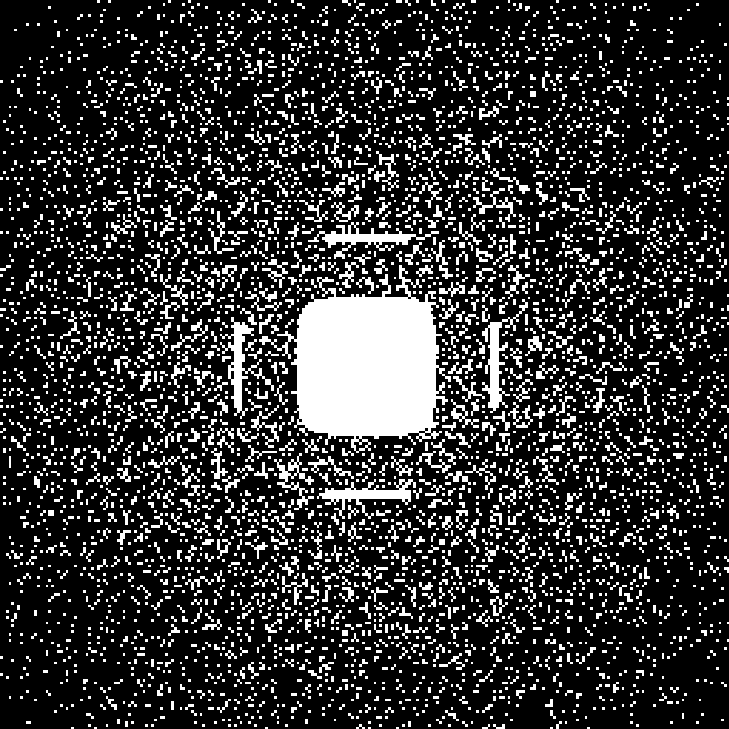}&
\includegraphics[height=0.3 \linewidth]{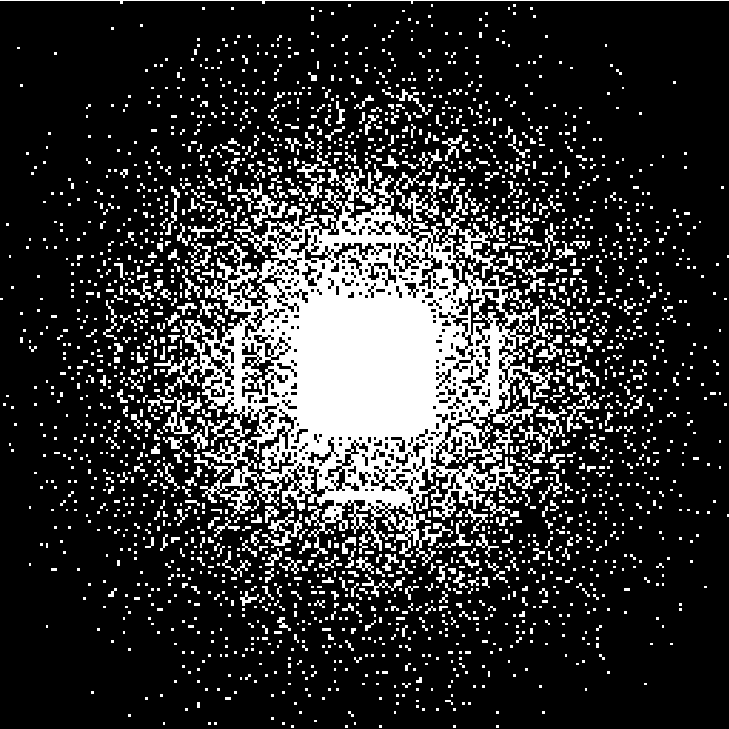}\\
{\small (d)}&{\small (e)}&{\small (f)} \\
\includegraphics[height=0.3 \linewidth]{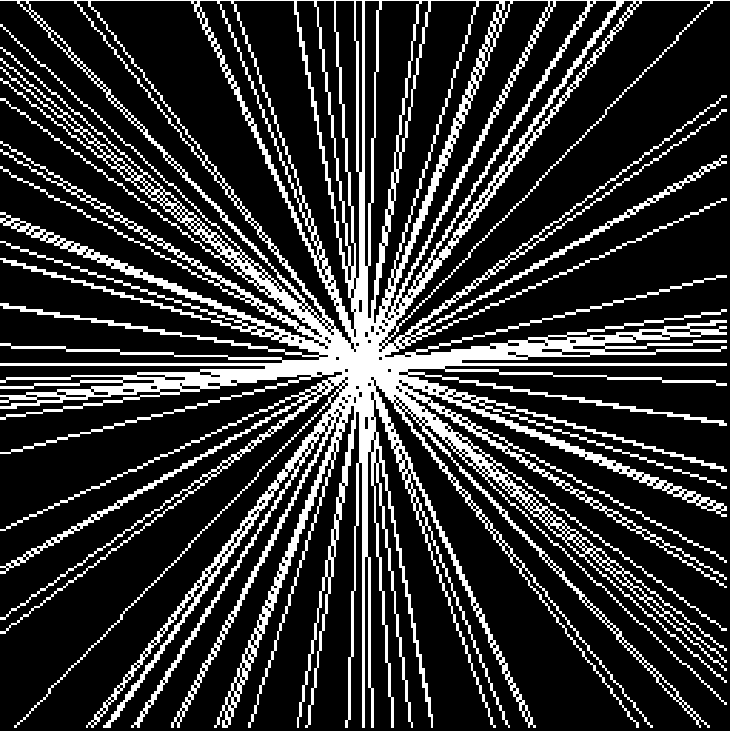}&
\includegraphics[height=0.3 \linewidth]{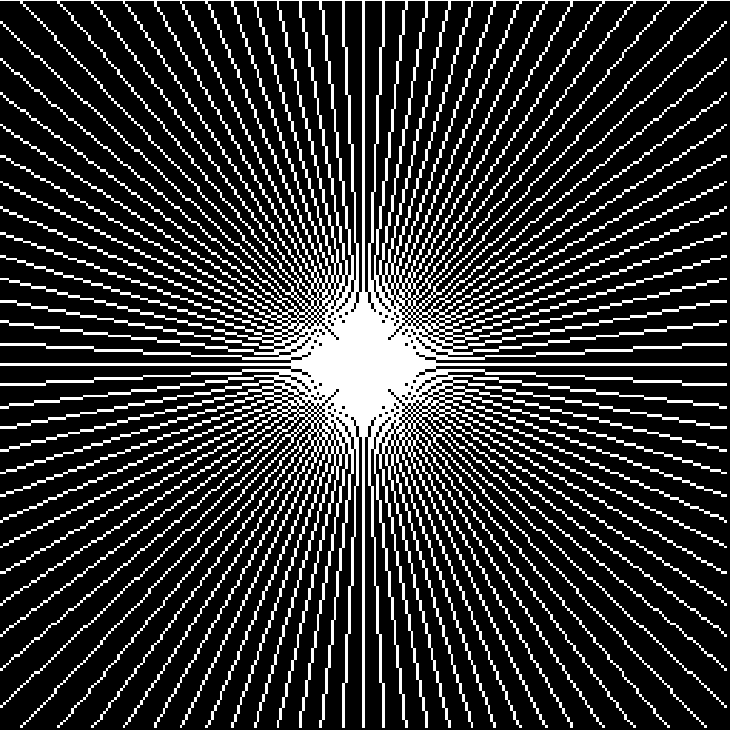}&
\includegraphics[height=0.3 \linewidth]{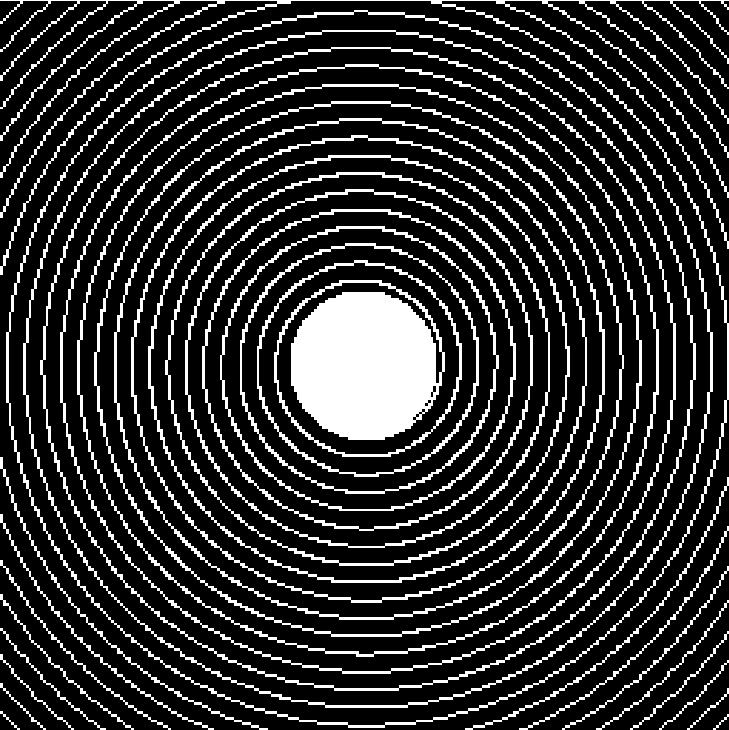}\\
\end{tabular}\vspace*{-.6cm}
\end{center}
\caption{\label{fig:sampling_schemes} \textbf{Various \kS sampling schemes.} (a)-(c): {\em Two-stage sampling schemes}. High frequencies sampled according to $\pi^\star$~(a) and to polynomial densities \cite{Lustig07} with $p=1$~(b) and $p=4$~(c).
(d)-(f): {\em Classical sampling schemes}. (d): Radial sampling with independent drawing of angles between lines. (e): Uniform radial sampling. (f): Spiral sampling, with a complete acquisition of low frequencies.} 
\end{figure}

\begin{table}[!h]
\caption{\label{tab:results} Reconstruction results in terms of mean PSNR~(and Std. dev. for random schemes) for the different patterns introduced in Fig.~\ref{fig:sampling_schemes}.
{\bf Bold font} indicate the best performance wrt the PSNR.}\vspace*{-.1cm}
\begin{center}
\begin{tabular}{|c|c||c|c|}
\hline
\multicolumn{2}{|c||}{Sampling density} & Mean PSNR (dB)& Std. dev.\\ 
\hline
\hline
\multirow{7}{0.7cm}{\begin{sideways}\parbox{15mm}{Two-stage strategies:}\end{sideways}}
& $\pi^\star$ & 35.87 & 0.08 \\ \cline{2-4}
& Polynomial (1) & 35.03 & 0.07\\ \cline{2-4}
& Polynomial (2) & 35.94 & 0.08\\ \cline{2-4}
& Polynomial (3) & \textbf{36.30} & 0.08\\ \cline{2-4}
& Polynomial (4) & \textbf{36.32} & 0.15\\ \cline{2-4}
& Polynomial (5) & \textbf{36.24} & 0.27\\ \cline{2-4}
& Polynomial (6) & 35.86 & 0.05\\ \cline{2-4}\hline \hline
\multicolumn{2}{|c||}{Random radial} & 31.97 & 0.52\\  \hline
\multicolumn{2}{|c||}{Radial} & 34.22 & \\  \hline
\multicolumn{2}{|c||}{Spiral} & 31.15 & \\  \hline
\end{tabular}
\end{center}
\vspace*{-0.7cm}
\end{table}

As shown in Tab.~\ref{tab:results}, sampling the whole \kS center drastically improves the reconstruction performance in comparison with Tab.~\ref{tab:prelim} irrespective of the inital approach. Moreover, in contrast to what we observed in Section~\ref{subsec:prelim_results} with $\pi$, our two-stage approach based on $\pi^\star$ achieves very close reconstruction results~(while less accurate by 0.45dB) to those based on a similar two-step procedure relying on high order polynomial densities~\cite{Lustig07}. As expected, the gain $g_p$ in dB for low order polynomial densities is larger~($g_p\in(4.3, 11.48)$~dB) for $p=1:3$) since increasing the exponent in such densities makes the sampling more dense around the $k$-space center. For the highest order density~($p=6$), we even observe a loss of PSNR indicating that our approach fails in this context. On the other hand, the gain we obtained for $\pi^\star$ is more striking since $g_{\pi}=2.49$~dB. Also, the standard deviation measured for all random schemes strongly decreases using our two-step approach. Besides, the two-stage sampling schemes perform better than spiral and radial patterns irrespective of the strategy with respect to angles~(regular or random). Nevertheless, continuity is required for practical implementation of MRI sequences and we are currently investigating the design of optimal continuous trajectories. Our aim is to derive continuous sampling patterns from conciliating optimal drawing and continuity using Markov chains.

The reasons for which the reconstruction results based on the optimal distribution $\pi^\star$ are less accurate than those based on polynomial density sampling are still unclear. Two main hypotheses formulate as follows: i.)~even after the first step (subtracting low frequencies), the signal would not be sparse enough preventing us from meeting the conditions of Theorem.~\ref{thm:Rauhut}. ii.) the bounds given in Theorem.~\ref{thm:Rauhut} might not be optimal and minimizing these bounds would not actually give the optimal sampling distribution. In particular, real images have a level of sparsity that increases among wavelet subbands, and this property should be taken into account in order to derive theoretical reconstruction results.

Although this theoretical approach does not seem optimal in terms of reconstruction performance, it encourages the use of compressed sensing for MRI. Indeed, this theory tends to sample the low frequencies more densely than the high frequencies, from considerations on the transform basis. The heuristics which have the best practical results~\cite{Lustig07} propose similar sampling density profiles. Theories~\cite{Rauhut10,CandesP11} suit with the fact that MRI images mainly contain low frequency information, though theoretical framework does not take this observation into consideration.
    

\section{CONCLUSION AND FUTURE WORK}
We proposed a two-stage method to design $k$-space sampling schemes. First, the signal is sparsified by acquiring the whole center of the \kS in order to recover the wavelet low frequency coefficients. Second, an independent random sampling of high frequencies is performed according to an optimal distribution for a given criterion so as to recover the remaining wavelet coefficients. Our results are comparable to the state-of-the-art. Also, we have shown that sampling the whole low frequencies drastically improves reconstruction quality. 
Our method is general enough to be used for any acquisition system measuring a sparse signal, which does not perfectly fulfill the sparsity hypothesis. In MRI, it seems that sparsity in the wavelet domain depends on the subbands. An interesting outlook of this work would be to include this remark in the theoretical framework.


%
%


{\footnotesize
\bibliographystyle{IEEEbib}
\bibliography{bibenabr,revuedef,revueabr,isbi}
}
\end{document}